\documentclass[preprint, 12pt]{elsarticle}
\usepackage{mathrsfs}
\usepackage{amsfonts}
\usepackage{amssymb}
\usepackage{amssymb,amsmath}
 \usepackage{graphicx}
 \usepackage{subfigure}

 \usepackage{latexsym}
\usepackage{amssymb}
\usepackage{amsmath}
\usepackage{amsthm}
\usepackage{verbatim}
\usepackage{mathrsfs}

\numberwithin{equation}{section}

\newtheorem{theorem}{Theorem}[section]
\newtheorem{lemma}{Lemma}[section]
\newtheorem{proposition}{Proposition}[section]
\newtheorem{remark}{Remark}[section]
\newtheorem{definition}{Definition}[section]

\allowdisplaybreaks
\journal{Sci. China Math.}
\begin{document}
\begin{frontmatter}

\title{Periodic solutions of a semilinear variable coefficient wave equation under asymptotic nonresonance conditions}
\author[ad1]{Hui Wei}
\ead{weihui01@163.com}
\author[ad2,ad3]{Shuguan Ji\corref{cor}}
\ead{jishuguan@hotmail.com}
\address[ad1]{Department of Mathematics, Luoyang Normal University, Luoyang 471934, P.R. China}
\address[ad2]{School of Mathematics and Statistics and Center for Mathematics and Interdisciplinary Sciences, Northeast Normal University, Changchun 130024, P.R. China}
\address[ad3]{School of Mathematics, Jilin University, Changchun 130012, P.R. China}
\cortext[cor]{Corresponding author.}

\begin{abstract}
We consider the periodic solutions of a semilinear variable coefficient wave equation arising from the forced vibrations of a nonhomogeneous string and the propagation of seismic waves in nonisotropic media.
The variable coefficient characterizes the inhomogeneity of media and its presence usually leads to the destruction of the compactness of the inverse of linear wave operator with periodic-Dirichlet boundary conditions on its range. In the pioneering work of Barbu and Pavel \cite{Barbu97a}, it gives the existence and regularity of periodic solution for Lipschitz, nonresonant and monotone nonlinearity under the assumption $\eta_u>0$ (see Sect. 2 for its definition) on the coefficient $u(x)$ and leaves the case $\eta_u=0$ as an open problem. In this paper,
by developing the invariant subspace method and using the complete reduction technique and Leray-Schauder theory, we obtain the existence of periodic solutions for such a problem when the nonlinear term satisfies the asymptotic nonresonance conditions. Our result not only does not need any requirements on the coefficient except for the natural  positivity assumption (i.e., $u(x)>0$), but also does not need the monotonicity assumption on the nonlinearity.  In particular, when the nonlinear term is an odd function and satisfies the global nonresonance conditions, there is only one (trivial) solution to this problem on the invariant subspace.

\end{abstract}


\end{frontmatter}

\section{Introduction}
In the present paper, our concern is the existence of the periodic solutions of the semilinear variable coefficient wave equation
\begin{equation}
\label{Eqn1-1}
u(x) y_{tt} -  (u(x)y_x)_x = f(t,x,y),  \ \  (t, x) \in \Omega := (0, T) \times (0, \pi),
\end{equation}
with the periodic conditions
\begin{equation}
\label{Eqn1-2}
y(0,x) = y(T,x),  \ \ y_t(0,x) = y_t(T,x),  \ \   x\in(0, \pi),
\end{equation}
and the Dirichlet boundary conditions
\begin{equation}
\label{Eqn1-3}
y(t,0) =  y(t,\pi) = 0, \ \  t\in (0, T),
\end{equation}
where the nonlinear term $f(t,x,y) \in C(\Omega\times \mathbb{R}, \mathbb{R})$ is periodic in time $t$ with a prescribed
period $T$ satisfying
\begin{equation}
\label{Eqn1-4}
T = 2\pi \frac{p}{q},
\end{equation}
with $p, q \in \mathbb{N}^+=\{1, 2, 3, \cdots\}$ and GCD $(p, q)=1$.

Equation (\ref{Eqn1-1}) is a mathematical model to describe the forced vibrations of a nonhomogeneous string and the propagation of seismic waves in nonisotropic media
(see e.g. \cite{Barbu97a,Ji06,W.19,WJ.19}). The variable coefficient $u=(\rho\mu)^{1/2}$, which is called the acoustic impedance function, arises naturally when the vibration equation $\rho(l)y_{tt}-(\mu(l)y_z)_z=0$ (with variable density $\rho$ and elasticity coefficient $\mu$) is normalized into $u(x)y_{tt}-(u(x)y_x)_x=0$ via the change of variable $x=\int_0^l(\frac{\rho(s)}{\mu(s)})^{1/2}{\rm d}s$.

It is obvious that the case of $u(x)\equiv 1$ corresponds to the classical (one-dimensional) wave equation. The problem of finding periodic solutions of such equations has been widely considered since the pioneer work of Rabinowitz \cite{Ra67}.
For example, see \cite{Berti03,B.81,B.78a,BW94,Craig93,Ra78} for one dimensional case and \cite{Ben95,Berti10,Ben93,Chen16,W.18} for higher dimensional case. Most of these results are based on the  spectral properties of wave operator. In particular, for the  Dirichlet boundary value problems, when the time period $T$ is a rational multiple of string length $\pi$,  then $0$ is the eigenvalue with infinite multiplicity of the wave operator $\partial_{tt}-\partial_{xx}$,
and the remaining eigenvalues  are well separated and accumulate to infinity. This good separation properties implies that the inverse of $\partial_{tt}-\partial_{xx}$  is  compact on its range. 
So the compactness method can be applied to estimate the component in range space of a periodic solution.
To estimate the component in kernel space, in general it needs to require that the nonlinearity is monotone. Thus, the nonlinear problems can be solved by using compactness and monotonicity methods (see e.g. \cite{Chang82,Ding98,Fokam17,Tanaka88,Tanaka06}).
On the other hand, when $T$ is an irrational multiple of $\pi$, then $0$ is an accumulation point of the spectrum,  thus it involves the small divisor problem and the variational method does not work, and so a quite different approach, the infinite dimensional KAM theory, is proposed to deal with such problems (see \cite{Ku87, W.90}).

For the variable coefficient wave equations, Barbu and Pavel \cite{Barbu97a} established firstly the existence and regularity of periodic solution for Lipschitz, nonresonant and monotone nonlinearity under the assumption $\eta_{u}(x)= \frac{1}{2} \frac{u''}{u}-\frac{1}{4}\big(\frac{u'}{u}\big)^2>0$ on the coefficient $u(x)$ and left the case $\eta_u=0$ as an open problem. Subsequently, the research on periodic solutions of variable coefficient wave equations has gained more attention (see e.g. \cite{Chen15,Ji08,Ji07,Ji16,Ma18,Ru16,Ru17}).
It is worth mentioning that, in \cite{Ji11}, Ji and Li  obtained the existence of periodic solution of the nonlinear wave equation with sublinear nonlinearity and constant or variable coefficients which may not satisfy the assumption $\eta_{u}(x)>0$. Recently, Ji \cite{Ji.18} found that, for some types of boundary value conditions but not including Dirichlet boundary value condition, the inverse of variable coefficient wave operator is still compact on its range. Thus, without imposing the the assumption $\eta_u(x) >0$, he got the existence of periodic solutions for the monotone and bounded nonlinearity.

In this paper, we study the periodic solutions of the variable coefficient wave equation (\ref{Eqn1-1}) with the periodic conditions (\ref{Eqn1-2}) and the Dirichlet boundary value conditions (\ref{Eqn1-3}). This is one type of the most basic boundary conditions, but was left unsolved in \cite{Ji.18} due to the loss of compactness for the inverse of variable coefficient wave operator. Here, by a further  research on the spectrum of the variable coefficient wave operator, we construct an invariant subspace in which the spectrum has good separation properties. Then, under some symmetry condition on the nonlinear term, the problem \eqref{Eqn1-1}--(\ref{Eqn1-3}) can be reduced onto this invariant subspace. Finally, by employing Leray-Schauder type method, we obtain the existence of  periodic solution to the problem \eqref{Eqn1-1}--\eqref{Eqn1-3} when the nonlinear term satisfies the asymptotic nonresonance conditions in a nonuniform manner with respect to $t$ and $x$. The result we obtained does not require any restrictive conditions on the coefficient except for the natural  positivity assumption $u(x)>0$, and it does not need the monotonicity assumption on the nonlinearity. In particular, when the nonlinear term is odd with respect to $y$ and satisfies the global nonresonance conditions (see \eqref{eq3.9}), the problem \eqref{Eqn1-1}--\eqref{Eqn1-3} has only one (trivial) solution on the invariant subspace.

Denote $$\hat{f}(t,x,y)={f(t,x,y)}/{u(x)}.$$
 Throughout this paper we make the following assumptions:

(H1) $u(x) \in C^2[0, \pi]$,  $u(x)>0$ for all $x\in [0,\pi]$.

(H2) For every constant $R>0$, there exists $h_R(t,x) \in L^2(\Omega)$ such that
\begin{equation}\label{eq1.4}
|\hat{f}(t,x,y)| \leq h_R(t,x), \ \ {\rm for} \ |y| \leq R.
\end{equation}

An outline of this paper is organized as follows.  In Sect. \ref{sec:2}, we introduce the definition of weak periodic solution and characterize some properties of the spectrum of the variable coefficient wave operator. Then, two invariant subspaces of $L^2(\Omega)$ are constructed and
some fundamental lemmas are given in  Sect. \ref{sec:3}. Finally, we establish and prove the main results (Theorems 4.1 and 4.2) in Sect. \ref{sec:4}.

\section{Preliminaries}
\setcounter{equation}{0}
\label{sec:2}

Denote
\begin{eqnarray*}
\Psi = \{\psi \in C^\infty(\Omega) \mid \psi(t,0) = \psi(t,\pi) = 0,  \psi(0,x) = \psi(T,x), \psi_t(0,x) = \psi_t(T,x) \},
\end{eqnarray*}
and define
$$L^r(\Omega) = \Big\{ y\mid \|y\|_{L^r(\Omega)} = \Big(\int_\Omega u(x) |y(t,x)|^r  \textrm{d}t \textrm{d}x \Big)^{\frac{1}{r}}<\infty\Big\}.$$
In particular, for $r=2$, $L^2(\Omega)$ is a Hilbert space with the inner product
$$\langle y, z \rangle = \int_\Omega u(x) y(t,x)  \overline{z(t,x)} \textrm{d}t \textrm{d}x, \ \forall y, z \in L^2(\Omega).$$
For the sake of convenience,  in what follows we use $\|\cdot\|$ to denote $\|\cdot\|_{L^2(\Omega)}$.

\begin{definition}
A function $y \in L^2(\Omega)$ is called a weak solution of problem \eqref{Eqn1-1}--\eqref{Eqn1-3} if it satisfies
$$\int_\Omega y(u(x)\psi_{tt} - (u(x)\psi_x)_x){\rm d}t \textrm{\rm d}x - \int_\Omega f(t,x,y) \psi \textrm{\rm d}t \textrm{\rm d}x = 0, \ \ \forall \psi \in \Psi. $$
\end{definition}


For the study of the periodic solution of problem  \eqref{Eqn1-1}--\eqref{Eqn1-3}, we need to analyze the spectrum of variable coefficient wave operator, which is closely related to the following Sturm-Liouville problem
\begin{eqnarray*}
&&(u(x)\varphi'_n(x))' = -\lambda_n^2 u(x) \varphi_n(x),  \ \ n\in \mathbb{N}^+,\\
&&\varphi_n(0) = \varphi_n(\pi) =0.
\end{eqnarray*}
In virtue of Sect. 4 in \cite{Fulton94}, it is known that
\begin{equation*}
\lambda_n^2= \left( n + \frac{\kappa}{2n\pi}  + O\left(\frac{1}{n^3}\right)\right)^2,
\end{equation*}
with
\begin{equation*}
\kappa = \int_0^\pi \eta_{u}(x) {\rm d}x, \ \ \eta_{u}(x) = \frac{1}{2} \frac{u''}{u}-\frac{1}{4} \left(\frac{u'}{u}\right)^2.
\end{equation*}
A direct calculation shows that
\begin{equation*}
\lambda_n^2=n^2 +\frac{\kappa}{\pi} + O(\frac{1}{n^2}), \ \ n\in \mathbb{N}^+.
\end{equation*}

Assume that $\varphi_n(x)$ is the corresponding normalized eigenfunction which satisfies $\int_0^\pi u(x)|\varphi_n(x)|^2{\rm d}x=1$. Then the functions system
$$\left\{T_m\varphi_n\cos \frac{2\pi}{T}mt, T_m\varphi_n\sin \frac{2\pi}{T}mt\right\}, \ \ m \in\mathbb{N}=\mathbb{N}^+\cup \{0\}, n \in \mathbb{N}^+, $$
with $T_0= 1/\sqrt{T}$ and $T_m = \sqrt{2/T}$ for $m \in \mathbb{N}^+$, is completely orthonormal in $L^2(\Omega)$.
Thus, in view of $T=2\pi\frac{p}{q}$ in (\ref{Eqn1-4}), a function $y \in L^2(\Omega)$ can be written as the Fourier series
\begin{equation}
\label{Eqn2-1}
y=\sum_{n=1}^{\infty}\sum_{m=0}^{\infty} T_m \varphi_n(x) \left( a_{mn} \cos \frac{q}{p}mt + b_{mn} \sin \frac{q}{p}mt\right),
\end{equation}
with
\begin{eqnarray*}
&&a_{mn}=T_m\int_\Omega u(x)\varphi_n(x) y(t,x)  \cos \frac{q}{p}mt {\rm d}t{\rm d}x,\\
&&b_{mn}=T_m\int_\Omega u(x)\varphi_n(x) y(t,x)  \sin \frac{q}{p}mt {\rm d}t{\rm d}x,
\end{eqnarray*}
for $ m \in\mathbb{N}, n \in \mathbb{N}^+$.

Set
$$ \mathcal{D}(L)=\left\{y\in L^2(\Omega)\Big| \sum_{n=1}^{\infty}\sum_{m=0}^{\infty} \Big( \lambda_n^2 - \big(\frac{q}{p}m\big)^2 \Big)^2(a_{mn}^2 + b_{mn}^2) < \infty\right\}.$$
Then, for any $y\in \mathcal{D}(L)$ with the Fourier expansion (\ref{Eqn2-1}), we define the operator $L: \mathcal{D}(L)\rightarrow L^2(\Omega)$ as
\begin{eqnarray*}
 y\mapsto Ly=\sum_{n=1}^{\infty}\sum_{m=0}^{\infty} \Big( \lambda_n^2 - \big(\frac{q}{p}m\big)^2 \Big)T_m \varphi_n(x) \left( a_{mn} \cos \frac{q}{p}mt + b_{mn} \sin \frac{q}{p}mt\right).
\end{eqnarray*}
It is obvious that
$$\mu_{mn}=\lambda_n^2 - \big({qm}/{p}\big)^2$$
 are the eigenvalues of $L$,  and its set is denoted by
$$\Lambda(L)=\{\mu_{mn} \mid m \in\mathbb{N}, \ n \in \mathbb{N}^+\}.$$

\begin{proposition}\label{pp2.1}
Assume $p$ is an even number. Then, for odd $m$, the eigenvalue $\mu_{mn}$ are isolated and has finite multiplicity.
\end{proposition}
\begin{proof}
If $p$ is even, then $q$ is odd. Thus, for odd $m$, we have $np \neq mq$. Therefore, it is easy to know
\begin{equation}\label{eq2-2}
\mu_{mn}=\lambda_n^2 - \big(\frac{q}{p}m\big)^2=\frac{(np-mq)(np+mq)}{p^2} +\frac{\kappa}{\pi} +O(\frac{1}{n^2})\rightarrow \infty,
\end{equation}
as $m, n \rightarrow \infty$. The proof is completed.
\end{proof}

\section{Invariant subspaces of $L^{2}(\Omega)$ and some fundamental lemmas}
\label{sec:3}

In this section, we assume $p$ is an even number. By Proposition \ref{pp2.1}, we may construct an invariant subspace of $L^2(\Omega)$ in which  the spectrum of the restricted operator has good separation properties.
Then we restrict the problem \eqref{Eqn1-1}--\eqref{Eqn1-3} on this invariant subspace and give several fundamental results.
\begin{definition}
\label{Def3-1}
Let $\mathcal{M}_1$ and $\mathcal {M}_2$  be the closed subspaces of $L^2(\Omega)$ such that $L^2(\Omega)=\mathcal{M}_1\oplus\mathcal{M}_2$
and $\mathcal{M}_1 \perp \mathcal{M}_2$. The operator $L$ is called to be completely reduced by $\mathcal{M}_1$ and $\mathcal{M}_2$ if it satisfies
$$P_i \big(\mathcal{D}(L)\big) \subset \mathcal{D}(L), \ \ LP_i(y)=P_i L(y), \ \ i=1,2,$$
for any $y \in \mathcal{D}(L)$, where $P_i$ is the projection onto $\mathcal{M}_i$ for $i=1, 2$.
\end{definition}
The restriction of $L$ on $\mathcal{M}_i$ is denoted by $L_i=L|_{\mathcal{M}_i\cap \mathcal{D}(L)}$ for $i=1,2$. Then $L_i$ inherits the properties of $L$ and $\Lambda(L_i)\subset \Lambda(L)$ for $i=1, 2$.

Define
\begin{equation*}
\mathcal{M}_o = \Big\{ y\in L^{2}(\Omega)\mid y(t,x)= -y(t+T/2, x)\Big\}
\end{equation*}
and
\begin{equation*}
\mathcal{M}_e= \Big\{ y\in L^{2}(\Omega)\mid y(t,x)=y(t+T/2, x)\Big\}.
\end{equation*}
It is not difficult to verify that
\begin{eqnarray*}
&&\mathcal{M}_o = {\rm {Span}}\Big\{ \varphi_n(x)\cos \frac{q}{p}mt,  \varphi_n(x)\sin \frac{q}{p}mt \mid  n \in \mathbb{N}^+, m \ {\rm is \ odd}\Big\},\\
&&\mathcal{M}_e = {\rm {Span}}\Big\{ \varphi_n(x)\cos \frac{q}{p}mt,  \varphi_n(x)\sin \frac{q}{p}mt \mid  n \in \mathbb{N}^+, m \ {\rm is \ even}\Big\},
\end{eqnarray*}
and $L$ is completely reduced by $\mathcal{M}_o$ and $\mathcal{M}_e$.

For simplicity,  denote
$$L_o=L|_{\mathcal{M}_o \cap \mathcal{D}(L)}\ \  {\rm and} \ \ L_e=L|_{\mathcal{M}_e \cap \mathcal{D}(L)}. $$
Then, by Proposition \ref{pp2.1}, it is known that $\dim (\ker L_o) < \infty$ and  $\Lambda(L_o)$ is an unbounded discrete set. Moreover, by a standard way (see e.g. \cite{Barbu97a, Ji.18}), it is easy to verify $\mathcal{M}_o = \ker L_o\oplus \mathcal{R}(L_o)$ and $L_o^{-1}$ is compact on $\mathcal{R}(L_o)$, where $\mathcal{R}(L_o)$ denotes the range of $L_o$.

\begin{lemma}\label{lem:3.1}
Let $\underline{\lambda}, \bar{\lambda} \in \Lambda(L_o)$ be two consecutive eigenvalues. Assume $\alpha$, $\beta \in \mathcal{M}_e \cap L^\infty(\Omega)$ satisfying
\begin{equation}\label{eq3-f}
\underline{\lambda} \leq \alpha(t,x) \leq \beta(t,x)\leq \bar{\lambda}, \ \ {for \ a.e.} \ (t,x)\in\Omega.
\end{equation}
Moreover, assume
\begin{equation}\label{eq3-g}
\int_\Omega(\alpha-\underline{\lambda})v^2{\rm d}t {\rm d}x>0, \ \ \forall v\in\ker(L_o-\underline{\lambda} I) \backslash \{0\},
\end{equation}
and
\begin{equation}\label{eq3-h}
\int_\Omega(\bar{\lambda}-\beta)w^2{\rm d}t {\rm d}x>0, \ \ \forall w\in\ker(L_o-\bar{\lambda} I) \backslash \{0\}.
\end{equation}
Then, there exist $\varepsilon >0$ and $\delta >0$ such that for any  $\gamma \in \mathcal{M}_e \cap L^\infty(\Omega)$ satisfying
\begin{equation}\label{eq3-a}
\alpha(t,x) -\varepsilon \leq \gamma (t,x)\leq \beta(t,x) +\varepsilon, \ \ { for \ a.e. \ } (t,x)\in \Omega,
\end{equation}
we have $\|L_o y - \gamma y\| \geq \delta \|y\|$ for all $y \in \mathcal{D}(L_o)\cap \mathcal{M}_o$.
\end{lemma}

\begin{remark}
Here the conditions $\alpha, \beta, \gamma\in \mathcal{M}_e$ are imposed to make sure  $\alpha y, \beta y, \gamma y\, \in \mathcal{M}_o$ for any $y \in \mathcal{M}_o$.
\end{remark}

\begin{proof}
 We prove the result by contradiction. In fact, if the result is false, we can find a sequence $\{y_j\} \subset \mathcal{D}(L_o)\cap \mathcal{M}_o$ with $\|y_j\|=1$  and $\{\gamma_j\} \subset \mathcal{M}_e \cap L^\infty(\Omega)$ such that
\begin{eqnarray}\label{eq-3.1}
\alpha(t,x) -\frac{1 }{j}\leq \gamma_j (t,x)\leq \beta(t,x) +\frac{1 }{j}, \ \ j\in  \mathbb{N}^+,
\end{eqnarray}
for a.e. $(t,x)\in \Omega$, and
\begin{equation*}
\gamma_j (t,x)\rightarrow \gamma_0 (t,x), \ \  {\rm as} \ \ j\rightarrow \infty,
\end{equation*}
for some $\gamma_0 \in \mathcal{M}_e \cap L^\infty(\Omega)$ and a.e. $(t,x)\in \Omega$,  and
\begin{equation}\label{eq3-z}
\|L_{o} y_j - \gamma_j y_j\| \leq \frac{1}{j}.
\end{equation}
Denote
\begin{equation}\label{eq-3.2}
\xi_j = L_{o} y_j - \gamma_j y_j.
\end{equation}
Then, we have $\|\xi_j\|\leq \frac{1}{j}$ for $j\in  \mathbb{N}^+$ and
\begin{equation}\label{eq-3.3}
\alpha(t,x) \leq \gamma_0 (t,x)\leq \beta(t,x),
\end{equation}
for a.e. $(t,x)\in \Omega$. 

Define
$$\mathcal{H}_1 = \Big\{ y\in \mathcal{M}_o\mid   \mu_{mn}\leq \underline{\lambda},\, m \in\mathbb{N}, \ n \in \mathbb{N}^+ \Big\}$$
and
$$\mathcal{H}_2 = \Big\{ y\in \mathcal{M}_o \mid   \mu_{mn}\geq \bar{\lambda}, \, m \in\mathbb{N}, \ n \in \mathbb{N}^+ \Big\}.$$
Since $\underline{\lambda}, \bar{\lambda}$ are two consecutive eigenvalues, we have
\begin{equation*}
\mathcal{M}_o = \mathcal{H}_1 \oplus \mathcal{H}_2.
\end{equation*}

Split $y_j=y_{1j} +y_{2j}$ with $y_{1j} \in \mathcal{D}(L_o)\cap \mathcal{H}_1$ and $y_{2j} \in \mathcal{D}(L_o)\cap \mathcal{H}_2$.
Let us  multiply both sides of equation \eqref{eq-3.2} by $y_{2j}-y_{1j}$, and then, by taking the inner product of $L^2(\Omega)$, it yields

\begin{equation*}
\langle L_o y_j - \gamma_j y_j,y_{2j}-y_{1j} \rangle =\langle \xi_j, y_{2j}-y_{1j}\rangle,
\end{equation*}
i.e.,
\begin{equation*}\label{eq2.6}
\langle L_o y_{2j} - \gamma_j y_{2j}, y_{2j}\rangle - \langle L_o y_{1j}- \gamma_j y_{1j}, y_{1j} \rangle =\langle \xi_j, y_{2j}-y_{1j}\rangle.
\end{equation*}

Let $\underline{\lambda}^* < \underline{\lambda}$ and $\bar{\lambda} <\bar{\lambda}^*$ be two pairs of consecutive eigenvalues, then  $\underline{\lambda}^* < \underline{\lambda}<\bar{\lambda} <\bar{\lambda}^*$.
Decompose
\begin{equation*}
\mathcal{H}_i = \mathcal{H}_i^* \oplus \mathcal{H}_i^0, \ \ i=1,2,
\end{equation*}
where $\mathcal{H}_1^*$ (resp. $\mathcal{H}_2^*$) and $\mathcal{H}_1^0$ (resp. $\mathcal{H}_2^0$) are spanned by the eigenfunctions  corresponding to the eigenvalues which satisfy $\mu_{mn} \leq \underline{\lambda}^*$ (resp. $\mu_{mn} \geq \bar{\lambda}^*$) and $\mu_{mn} = \underline{\lambda}$ (resp. $\mu_{mn} = \bar{\lambda}$), respectively. Consequently, we can rewrite
\begin{equation*}
y_{ij} = y_{ij}^* +y_{ij}^0, \ \ {\rm for} \ \ y_{ij}^*\in \mathcal{H}_i^*, \ y_{ij}^0 \in \mathcal{H}_i^0, \   i=1,2.
\end{equation*}

According to  $\|y_j\|^2=\|y_{1j}\|^2+\|y_{2j}\|^2=1$, we have
\begin{equation*}
\|y_{1j}- y_{2j}\|\leq \|y_{1j}\|+\|y_{2j}\| \leq 2 \Big(\frac{\|y_{1j}\|^2+\|y_{2j}\|^2}{2}\Big)^{\frac{1}{2}} \leq \sqrt{2}.
\end{equation*}
Thus, we have
\begin{eqnarray}
\frac{\sqrt{2}}{j}&\geq &\langle \xi_j, y_{2j}-y_{1j}\rangle \nonumber\\
&=&\langle L_o y_{2j} - \gamma_j y_{2j}, y_{2j}\rangle - \langle L_o y_{1j}- \gamma_j y_{1j}, y_{1j} \rangle\nonumber\\
&=&\langle L_o (y_{2j}^* + y_{2j}^0) - \gamma_j (y_{2j}^* + y_{2j}^0), (y_{2j}^* + y_{2j}^0)\rangle\nonumber\\
&&-\langle L_o (y_{1j}^* + y_{1j}^0)
- \gamma_j (y_{1j}^* + y_{1j}^0), (y_{1j}^* + y_{1j}^0) \rangle \nonumber\\
&=& I_1 - I_2,\label{eq3-d}
\end{eqnarray}
with
\begin{eqnarray}
I_1 &=&  \langle L_o y_{2j}^*  - \gamma_j y_{2j}^* , y_{2j}^*\rangle + \langle (\bar{\lambda}  - \gamma_j) y_{2j}^0, y_{2j}^0\rangle, \label{eq-3.6}\\
I_2&=&  \langle L_o y_{1j}^*  - \gamma_j y_{1j}^* , y_{1j}^*\rangle + \langle (\underline{\lambda}  - \gamma_j)  y_{1j}^0, y_{1j}^0\rangle. \label{eq-3.7}
\end{eqnarray}
Here we make use of the orthogonality of $\mathcal{H}_i^*$ and $\mathcal{H}_i^0$.

By \eqref{eq3-f} and \eqref{eq-3.1}, we have
\begin{eqnarray}
I_1 &\geq& \langle L_o y_{2j}^*  - (\bar{\lambda} + \frac{1}{j}) y_{2j}^* , y_{2j}^*\rangle - \frac{1}{j}\langle y_{2j}^0, y_{2j}^0\rangle,\label{eq3-b}\\
I_2 &\leq& \langle L_o y_{1j}^*  - (\underline{\lambda} - \frac{1}{j})y_{1j}^* , y_{1j}^*\rangle + \frac{1}{j}\langle  y_{1j}^0, y_{1j}^0\rangle.\label{eq3-c}
\end{eqnarray}
Therefore, from \eqref{eq3-d}, \eqref{eq3-b} and \eqref{eq3-c}, we have
\begin{eqnarray*}
\frac{\sqrt{2}}{j}&\geq & \langle L_o y_{2j}^*  - \bar{\lambda} y_{2j}^* , y_{2j}^*\rangle  -\langle L_o y_{1j}^*  - \underline{\lambda} y_{1j}^* , y_{1j}^*\rangle - \frac{1}{j}\nonumber\\
&\geq & (\bar{\lambda}^*-\bar{\lambda})\|y_{2j}^*\|^2 + (\underline{\lambda} -\underline{\lambda}^*)\|y_{1j}^*\|^2 - \frac{1}{j}.
\end{eqnarray*}
Thus, it follows
\begin{equation}\label{eq-3.4}
y_{1j}^* \rightarrow 0, \ \ y_{2j}^* \rightarrow 0, \ \ {\rm as} \ j\rightarrow \infty.
\end{equation}

Since $\{y_j\}$ is a bounded sequence and $\dim \mathcal{H}_i^0 <\infty$ ($i=1,2$) provided by Proposition \ref{pp2.1}, then there exist $y_1^0 \in \mathcal{H}_1^0$ and $y_2^0 \in \mathcal{H}_2^0$ such that
\begin{equation*}
y_{1j}^0 \rightarrow y_1^0, \ \ y_{2j}^0 \rightarrow y_2^0, \ \ {\rm as} \ j\rightarrow \infty.
\end{equation*}
Thus,  in view of \eqref{eq-3.4} and $y_{ij} = y_{ij}^* +y_{ij}^0$  ($i=1,2$), we have
\begin{equation*}
y_{1j} \rightarrow y_1^0, \ \ y_{2j} \rightarrow y_2^0, \ \ {\rm as} \ j\rightarrow \infty,
\end{equation*}
and
\begin{equation*}
y_j \rightarrow y_{\infty}:= y_1^0 + y_2^0, \ \ {\rm as} \ j\rightarrow \infty,
\end{equation*}
with $\|y_{\infty}\|^2= \|y_1^0\|^2 + \|y_2^0\|^2=1$ provided by $\|y_j\|=1$.

According to \eqref{eq3-z}, we have $\|L_o y_{ij}^*\|$ ($i=1,2$) is bounded.
Thus, by \eqref{eq-3.4} and $\gamma_j (t,x)\rightarrow \gamma_0 (t,x)$, we have
\begin{equation*} \label{eq-3.8}
\langle L_o y_{2j}^*  - \gamma_j y_{2j}^* , y_{2j}^*\rangle \rightarrow 0 \ \ {\rm and} \ \ \langle L_o y_{1j}^*  - \gamma_j y_{1j}^* , y_{1j}^*\rangle \rightarrow 0, \ \ {\rm as} \ \ j\rightarrow \infty.
\end{equation*}
Consequently, from \eqref{eq3-d}--\eqref{eq-3.7}, we have
\begin{equation*}
\int_{\Omega} (\gamma_0-\underline{\lambda})|y_{1}^0|^2 {\rm d}t{\rm d}x + \int_{\Omega} (\bar{\lambda} -\gamma_0)|y_{2}^0|^2 {\rm d}t{\rm d}x  =0.
\end{equation*}
Therefore, by \eqref{eq3-f} and \eqref{eq-3.3}, it follows
\begin{equation*}\label{eq-3.9}
\int_{\Omega} (\gamma_0-\underline{\lambda})|y_{1}^0|^2 {\rm d}t{\rm d}x =0 \ \ \  {\rm and } \ \ \
\int_{\Omega} (\bar{\lambda} -\gamma_0)|y_{2}^0|^2 {\rm d}t{\rm d}x =0.
\end{equation*}

Now, set
\begin{equation*}
\Omega_i = \{(t,x) \in \Omega \mid y_{i}^0 \neq 0\}, \ \ {\rm for} \ \ i=1,2.
\end{equation*}
Thus, it follows
\begin{equation*}\label{eq-3.10}
 \gamma_0=\underline{\lambda}, \ {\rm on} \ \Omega_1, \ \ \ {\rm and}\ \ \  \gamma_0=\bar{\lambda}, \ {\rm on} \ \Omega_2.
\end{equation*}
Furthermore, with the help of  $\underline{\lambda} < \bar{\lambda}$, we have $\Omega_1 \cap \Omega_2 =\emptyset$.

We now focus the attention on the discussion of the two cases $\Omega_1= \emptyset$ and $\Omega_1\neq \emptyset$, and then  we shall encounter a contradiction and come to the desired conclusion.

If $\Omega_1= \emptyset$, we have $y_{\infty} = y_{2}^0$. Then
\begin{equation*}
0=\int_{\Omega} (\mu - \gamma_0)|y_{2}^0|^2{\rm d}t{\rm d}x \geq \int_{\Omega} (\mu - \beta)|y_{2}^0|^2{\rm d}t{\rm d}x.
\end{equation*}
By \eqref{eq3-h}, we have $y_{\infty} = y_{2}^0 =0$ which contradicts with $\|y_{\infty}\| =1$.

If $\Omega_1\neq \emptyset$, we have
\begin{eqnarray*}
0= \int_{\Omega} (\gamma_0-\lambda)|y_{1}^0|^2 {\rm d}t{\rm d}x &=&\int_{\Omega_1} (\gamma_0-\lambda)|y_{1}^0|^2 {\rm d}t{\rm d}x\\
&\geq& \int_{\Omega_1} (\alpha-\lambda)|y_{1}^0|^2 {\rm d}t{\rm d}x = \int_{\Omega} (\alpha-\lambda)|y_{1}^0|^2 {\rm d}t{\rm d}x.
\end{eqnarray*}
By \eqref{eq3-g}, we have $y_{1}^0 =0$ on
$\Omega_1$, which contradicts with the definition of $\Omega_1$.

Thus, for either $\Omega_1= \emptyset$ or $\Omega_1\neq \emptyset$, we derive a contradiction, so the conclusion is true.
\end{proof}

\begin{proposition}[\cite{Ma79}]\label{pp4.1}
Let $\mathbf{X}$ and $\mathbf{Y}$ be two real normed vector spaces, and let $\Omega \subset \mathbf{X}$ be an open bounded set. Assume that $F=L+G$, $L : \mathcal{D}(L) \subset \mathbf{X}\rightarrow \mathbf{Y}$ is a linear Fredholm operator with zero index and $G: \mathbf{X}\rightarrow \mathbf{Y}$ is a linear and $L$-completely continuous mapping, where $\mathcal{D}(L)$ is the domain of $L$. If $\ker F$ is trivial and $0\in \Omega$, then
\begin{equation*}
D_L(F,\Omega) =\pm 1,
\end{equation*}
where $D_{L} (F, \Omega)$ denotes the degree of $F$ in $\Omega$ relative to $L$.
\end{proposition}
\begin{lemma}\label{lem:3.2}
Let $\gamma \in \mathcal{M}_e \cap L^\infty(\Omega)$ satisfy \eqref{eq3-a}.  Then, for every open ball $B_r\subset \mathcal{M}_o$ with center $0$ and radius $r>0$,
\begin{equation*}
D_{L_o} (L_o-\gamma I, B_r) =\pm 1,
\end{equation*}
where $I$ is the identity mapping.
\end{lemma}
\begin{proof}
 Noting that $\dim(\ker L_o)<\infty$, $\mathcal{M}_o = \ker L_o\oplus \mathcal{R}(L_o)$ and $L_o^{-1}$ is compact on $\mathcal{R}(L_o)$, we have that $L_o$ is a Fredholm operator with zero index and $\gamma I$ is $L_o$-completely continuous on $\mathcal{M}_o$.


By Lemma \ref{lem:3.1}, the equation
\begin{equation*}
L_oy - \gamma y =0
\end{equation*}
has only trivial solution in $\mathcal{D}(L_o)\cap\mathcal{M}_o$, which implies $\ker \big(L_o - \gamma I\big)$ is trivial. Thus, for every open ball $B_r\subset \mathcal{M}_o$ with center $0$ and radius $r$,
by Proposition \ref{pp4.1}, we have
\begin{equation*}
D_{L_o} (L_o-\gamma I, B_r) = \pm 1.
\end{equation*}
The proof is completed.
\end{proof}

\section{The main results}
\label{sec:4}

To prove the main results, we introduce a continuation theorem of the Leray-Schauder type.
\begin{lemma}[\cite{Ma79}] \label{th4.1}
Let $\mathbf{X}$ and $\mathbf{Y}$ be two real normed vector spaces. Assume that $L : \mathcal{D}(L) \subset \mathbf{X}\rightarrow \mathbf{Y}$ is a linear Fredholm operator with zero index and $N : \bar{B}_r \rightarrow \mathbf{Y}$ is a  $L$-compact operator, where $\mathcal{D}(L)$ is the domain of $L$ and $B_r \subset \mathbf{X}$ is an open ball with center $0$ and radius $r$. If there exist a constant $r>0$ and a $L$-compact operator $A : \mathbf{X} \rightarrow \mathbf{Y}$ such that

 {\rm (i)} for every $(y,s) \in (\mathcal{D}(L) \cap \partial B_r) \times (0,1)$;
\begin{equation*}
 Ly - (1-s)Ay -s Ny\neq 0,
\end{equation*}

{\rm (ii)} $0\not\in (L-A)(\mathcal{D}(L) \cap \partial B_r) $;

{\rm (iii)} $D_L(L-A, B_r) \neq 0$.\\
Then, the equation
\begin{equation*}
Ly -Ny =0
\end{equation*}
possesses at least one solution in $\mathcal{D}(L) \cap  \bar{B}_r$.
\end{lemma}

Let $F : L^2(\Omega) \rightarrow L^2(\Omega)$ be the  Nemytskii operator induced by $\hat{f}$, i.e.,
$F(y)(t,x)=\hat{f}(t,x,y)$.

To deal with the problem \eqref{Eqn1-1}--\eqref{Eqn1-3} on the subspace $\mathcal{M}_o$, it is necessary to require  $F(\mathcal{M}_o) \subset \mathcal{M}_o$. By a direct verification, we have $F(\mathcal{M}_o) \subset \mathcal{M}_o$ if and only if $\hat{f} =\hat{f}_1 + \hat{f}_2$, where $\hat{f}_1$ and $\hat{f}_2$ satisfy
\begin{equation}\label{eq3.3}
\left\{
\begin{array}{ll}
\hat{f}_1(t,x,y)=\hat{f}_1(t+T/2,x,y),\\
\hat{f}_1(t,x,y)=-\hat{f}_1(t,x,-y),
\end{array}
\right.
 \ {\rm and} \
\left\{
\begin{array}{ll}
\hat{f}_2(t,x,y)=-\hat{f}_2(t+T/2,x,y),\\
\hat{f}_2(t,x,y)=\hat{f}_2(t,x,-y).
\end{array}
\right.
\end{equation}
Since $\hat{f}_1$ is odd with respect to $y$, it is obvious that $y^{-1}\hat{f}_1(t,x,y)$  has same asymptotic behaviour for $y\rightarrow +\infty$ and $y\rightarrow -\infty$.

\begin{theorem}\label{th3.1}
  Assume that (H1)--(H2) hold,  $p$ is an even number, $\underline{\lambda}$, $\bar{\lambda} \in \Lambda(L_o)$ are two consecutive eigenvalues, and $\alpha, \beta \in \mathcal{M}_e \cap L^\infty(\Omega)$ satisfy \eqref{eq3-f}--\eqref{eq3-h}. If $\hat{f} =\hat{f}_1 + \hat{f}_2$, $\hat{f}_1$ and $\hat{f}_2$ satisfy the symmetry conditions \eqref{eq3.3}, and satisfy
\begin{equation}
\label{eq3.4}
\lim_{|y|\rightarrow +\infty} \frac{\hat{f_2}(t,x,y)}{y} = 0,
\end{equation}
uniformly  for a.e. $(t,x)\in\Omega$, and
\begin{equation}
\label{eq3.5}
\alpha(t,x)\leq  \liminf_{|y|\rightarrow +\infty} \frac{\hat{f}_1(t,x,y)}{y} \leq \limsup_{|y|\rightarrow +\infty} \frac{\hat{f}_1(t,x,y)}{y}\leq \beta (t,x),
\end{equation}
uniformly  for a.e. $(t,x)\in\Omega$.  Then problem \eqref{Eqn1-1}--\eqref{Eqn1-3} has at least one periodic solution  lying in $\mathcal{M}_o$.
\end{theorem}

\begin{proof}
Since $\hat{f} =\hat{f}_1 + \hat{f}_2$ and $\hat{f}_1$ and $\hat{f}_2$ satisfy \eqref{eq3.3}, $F(\mathcal{M}_o) \subset \mathcal{M}_o$, Thus, we consider problem \eqref{Eqn1-1}--\eqref{Eqn1-3} in the subspace $\mathcal{M}_o$.

By  \eqref{eq3.4} and \eqref{eq3.5}, there is $R>0$ such that
\begin{equation}\label{eq3-6}
\alpha(t,x)-\varepsilon \leq y^{-1}\hat{f}(t,x,y) \leq \beta(t,x) +\varepsilon, \ \ {\rm for} \ |y|\geq R,
\end{equation}
where $\varepsilon$ is presented in Lemma \ref{lem:3.1}.
The estimate \eqref{eq3-6} combining with (H2) gives
\begin{equation*}
\big|\hat{f}(t,x,y)\big|\leq C|y| + h_R(t,x),
\end{equation*}
for all $y\in \mathbb{R}$  and a.e. $(t,x)\in\Omega$, where the constant $C>0$ depends on $\varepsilon, R$. Thus, the operator
$F_o :=F|_{\mathcal{M}_o}$
is continuous and maps a bounded set into a bounded set. Then $y \in \mathcal{D}(L_o)\cap \mathcal{M}_o$ is a weak solution of problem \eqref{Eqn1-1}--\eqref{Eqn1-3} if and only if
\begin{equation}\label{eq3.6}
L_oy - F_o y =0.
\end{equation}

Recalling  that $L_o$ is a Fredholm operator with zero index, $\mathcal{M}_o = \ker L_o\oplus \mathcal{R}(L_o)$ and $L_o^{-1}$ is compact on $\mathcal{R}(L_o)$, it follows that $F_o$ is $L_o$-compact.
Thus, by Lemma \ref{lem:3.2}, for every $r>0$, it follows
\begin{equation*}
D_{L_o} (L_o-\alpha I, B_r) = \pm 1.
\end{equation*}

In addition, by Lemma \ref{lem:3.1}, $0\not\in (L_o-\alpha I)(\mathcal{D}(L_o) \cap \partial B_r)$.
Consequently, by Lemma \ref{th4.1}, equation \eqref{eq3.6} will possess at least one solution if the set of all possible solutions of
\begin{equation}\label{eq3.7}
L_oy - (1-s)\alpha y - s F_o(y) = 0, \ \ s\in (0, 1),
\end{equation}
is a priori bounded independently of $s$. Set
\begin{equation*}
g(t,x,y)= \left\{
\begin{array}{ll}
y^{-1}\hat{f}(t,x,y),  \quad\quad\quad\quad\quad\quad\quad\quad\quad \ \ \ {\rm for} \ |y| \geq R,\\
R^{-1}\hat{f}(t,x,R)\frac{y}{R} + (1- \frac{y}{R}) \alpha (t,x),   \ \ \ \ {\rm for} \ 0\leq y < R,\\
R^{-1}\hat{f}(t,x,-R)\frac{y}{R} + (1+ \frac{y}{R}) \alpha (t,x),   \ \ {\rm for} \ -R< y \leq 0.
\end{array}
\right.
\end{equation*}
Obviously,
\begin{equation*}
\alpha (t,x) -\varepsilon \leq g(t,x,y) \leq \beta (t,x) +\varepsilon,
\end{equation*}
for all $y\in \mathbb{R}$ and a.e. $(t,x)\in\Omega$. Let $\theta (t,x,y) = \hat{f}(t,x,y) - g(t,x,y)y$. Thus, we have
\begin{equation*}
|\theta (t,x,y)| \leq 2 h_R(t,x) + |\alpha (t,x)|,
\end{equation*}
for all $y \in \mathbb{R}$ and a.e. $(t,x)\in \Omega$. Define
\begin{equation*}
(G(y)z)(t,x)= g(t,x,y)z, \ \  \forall y, z\in L^2(\Omega).
\end{equation*}
For every $y \in \mathcal{M}_o$, let $G_o(y)= G(y)|_{\mathcal{M}_o}$. Since $g(t,x,y(t,x))=g(t+T/2,x,y(t +T/2,x))$ for $y\in\mathcal{M}_o$, $G_o(y): \mathcal{M}_o \rightarrow \mathcal{M}_o$. Define $\Theta(y)(t,x)= \theta(t,x,y)$  for $y \in L^2(\Omega)$, and let $\Theta_o= \Theta|_{\mathcal{M}_o}$. Hence
\begin{equation*}
F_o(y)= G_o(y)y +\Theta_o(y), \ \  \forall y \in \mathcal{M}_o.
\end{equation*}
Thus equation \eqref{eq3.7} is equivalent to
\begin{equation*}
L_oy - \Big((1-s)\alpha I +sG_o(y) \Big)(y) = s \Theta_o(y), \ \ s\in (0, 1).
\end{equation*}
Let $\tilde{\gamma} = (1-s)\alpha  +sg(t,x,y)$, we have $\alpha (t,x) -\varepsilon \leq \tilde{\gamma} \leq \beta (t,x) +\varepsilon$. Thus,  with the aid of Lemma \ref{lem:3.1}, it follows
\begin{equation*}
\delta \|y\| \leq \|L_oy - \tilde{\gamma} y\| =\|s \Theta(y)\| \leq 2 \|h_R\| +\pi T\|\alpha\|_{L^{\infty}(\Omega)},
\end{equation*}
i.e., $ \|y\| \leq \frac{1}{\delta}(2 \|h_R\| +\pi T\|\alpha\|_{L^{\infty}(\Omega)})$. The proof is completed.
\end{proof}

Now, we consider that $\hat{f}$ has the properties
\begin{equation}\label{eq3-m}
\hat{f}(t,x,y)=\hat{f}(t+T/2,x,y) \ \ \ \ {\rm and} \ \ \ \ \hat{f}(t,x,-y)=-\hat{f}(t,x,y),
\end{equation}
which implies $\hat{f}_2\equiv0$ and $\hat{f} =\hat{f}_1$ for any $(t,x,y)\in  \Omega\times \mathbb{R}$. Obviously, $F(\mathcal{M}_o) \subset \mathcal{M}_o$ and  $0$ is a trivial solution. The following theorem shows that, under the global nonresonance conditions, problem \eqref{Eqn1-1}--\eqref{Eqn1-3} does not possess other solutions in the invariant subspace $\mathcal{M}_o$.

\begin{theorem}
Assume that (H1)--(H2) hold,  $p$ is an even number, $\underline{\lambda}$, $\bar{\lambda} \in \Lambda(L_o)$ are two consecutive eigenvalues, and $\alpha, \beta \in \mathcal{M}_e \cap L^\infty(\Omega)$ satisfy \eqref{eq3-f}--\eqref{eq3-h}. If $\hat{f}$ satisfies  \eqref{eq3-m} and
\begin{equation}\label{eq3.9}
\alpha(t,x) \leq \frac{\hat{f}(t,x,y) - \hat{f}(t,x,z)}{y-z} \leq \beta(t,x),
\end{equation}
for $y \neq z$ and a.e. $(t,x)\in \Omega$.
Then, problem \eqref{Eqn1-1}--\eqref{Eqn1-3} has only one (trivial) solution in $\mathcal{M}_o$.
\end{theorem}
\begin{proof}
Since $\hat{f}$ satisfies  \eqref{eq3-m}, $\hat{f}_2\equiv0$ and \eqref{eq3.4} is verified.
By \eqref{eq3.9}, the estimate \eqref{eq3.5} holds. Thus,  by Theorem \ref{th3.1},  there exists at least one periodic solution to problem \eqref{Eqn1-1}--\eqref{Eqn1-3} in $\mathcal{M}_o$.

Notice that $0$ is a trivial solution  to problem \eqref{Eqn1-1}--\eqref{Eqn1-3}.  Now we show problem \eqref{Eqn1-1}--\eqref{Eqn1-3} does not possess other solutions in $\mathcal{M}_0$ by contradiction. Suppose  $y_0 \neq 0$ is a solution to problem \eqref{Eqn1-1}--\eqref{Eqn1-3},  then $y_0$ satisfies
\begin{equation}\label{eq3.10}
L_oy_0 - \hat{f}(t,x,y_0)=0.
\end{equation}
Set $\xi(t,x,y_0)=y_0^{-1}\hat{f}(t,x,y_0) \in \mathcal{M}_e$, then by \eqref{eq3.9}, we have
\begin{equation*}
\alpha (t,x)- \varepsilon \leq \xi(t,x,y_0) \leq \beta(t,x)+\varepsilon,
\end{equation*}
for a.e. $(t,x) \in \Omega$, where $\varepsilon$ is presented in Lemma \ref{lem:3.1}. Therefore, by Lemma \ref{lem:3.1} and equation \eqref{eq3.10}, we have
\begin{equation*}
\delta\|y_0\| \leq \|L_oy_0 - \xi(t,x,y_0)y_0\|=0.
\end{equation*}
So we have $y_0=0$, and it yields a contradiction. This is the result we desire.
\end{proof}

\vskip 5mm

{\bf Acknowledgements} This work was partially supported by National Natural Science Foundation of China (Grant Nos. 12071065, 11671071 and 11871140).





\bibliographystyle{elsarticle-num}
\bibliography{<your-bib-database>}



\end{document}